\documentclass[10pt]{amsart}
\usepackage[margin=0.5in]{geometry}
\usepackage{graphicx}
\usepackage{color}
\usepackage{amsmath}
\usepackage{bm} 
\usepackage{enumitem} 
\usepackage{mathabx} 
\usepackage{mathrsfs}
\allowdisplaybreaks

\newcommand{\black}{\textcolor{black}}

\newcommand{\E}{\mathscr E}
\newcommand{\D}{\mathscr D}
\newcommand{\Q}{{\mathbb Q}}

\newcommand{\pa}{\partial}
\newcommand{\na}{\nabla}

\newtheorem{theorem}{Theorem}[section]
\newtheorem{lemma}[theorem]{Lemma}

\newtheorem{remark}[theorem]{Remark}

\usepackage[colorlinks=true,citecolor=magenta,linkcolor=magenta]{hyperref}

\begin{document}

\title[Trend to equilibrium of renormalized solutions to reaction-cross-diffusion systems]{Trend to equilibrium of renormalized solutions to reaction-cross-diffusion systems}

\author[E. S. Daus]{Esther S. Daus}
\address{Institute for Analysis and Scientific Computing, Vienna University of
	Technology, Wiedner Hauptstra\ss e 8--10, 1040 Wien, Austria}
\email{esther.daus@tuwien.ac.at}

\author[B. Q. Tang]{Bao Quoc Tang}
\address{Institute of Mathematics and Scientific Computing, University of Graz, 
  Heinrichstrasse 36, 8010 Graz, Austria}
\email{quoc.tang@uni-graz.at}

\date{\today}

\thanks{}

\begin{abstract}
	The convergence to equilibrium of renormalized solutions to reaction-cross-diffusion systems in a bounded domain under no-flux boundary conditions is studied. The reactions model complex balanced chemical reaction networks coming from mass-action kinetics and thus do not obey any growth condition, \textcolor{black}{while the diffusion matrix is of cross-diffusion type and hence nondiagonal and neither symmetric nor positive semi-definite, but the system admits a formal gradient-flow or entropy structure.
	The diffusion term generalizes the population model of Shigesada, Kawasaki and Teramoto to an arbitrary number of species.} 
	By showing that any renormalized solution satisfies the conservation of masses and a weak entropy-entropy production inequality, it can be proved under the assumption of no boundary equilibria that {\it all} renormalized solutions converge exponentially to the complex balanced equilibrium with a rate which is explicit up to a finite dimensional inequality.
\end{abstract}



\maketitle



\section{Introduction}
Multi-species systems appear in many applications in biology, physics and chemistry, and can be modeled by reaction-cross-diffusion systems.
We want to study the convergence to equilibrium of reaction-cross-diffusion systems with strongly growing reactions, where the system (without reactions) is of formal gradient-flow structure and thus admits an entropy estimate. But since the reactions do not obey any growth condition, 
 this estimate is not enough to define weak solutions, which motivates the study of renormalized solutions \textit{\`{a} la} J.~Fischer \cite{Fi15}. Our goal is to show that any renormalized solution satisfies the conservation of masses and a weak entropy-entropy production inequality, and consequently, under the assumption of no boundary equilibria, all renormalized solutions converge to equilibrium with an exponential rate which is explicit up to a finite dimensional inequality.

The convergence to equilibrium for reaction-diffusion systems with linear diffusion has been studied extensively, see e.g. \cite{CDF14,DeFe06, DeFe14} and references therein, while much less is known for nonlinear diffusion or cross diffusion, see \cite{FLT17} for a porous-medium type diffusion and \cite{DJT18} for Maxwell-Stefan diffusion. In this work, we study the convergence to equilibrium for a cross-diffusion model originally introduced by Shigesada, Kawasaki and Teramoto \cite{SKT79} in population dynamics. 
The existence of global weak solutions for this class of cross-diffusion models with at most linearly growing reactions
has been attracted a lot of attention recently by exploiting its formal gradient-flow structure, see \textit{e.g} \cite{CDJ18,ChJu04, ChJu06, DLM14, DLMT15,Jue15, Jue16}. Unfortunately, for strongly growing reactions \textcolor{black}{(such as chemical reactions)} this does not provide enough regularity to define weak solutions. Hence, the notion of renormalized solutions was introduced in \cite{CJ17} for reaction-cross-diffusion systems in analogy to \cite{Fi15} for reaction-diffusion systems. The standard way for proving convergence to equilibrium via entropy method is to first prove the convergence for an approximate solution, and then by passing to the limit to obtain it also for the {\it constructed} weak solution (see \textit{e.g.} \cite{DJT18}). But since uniqueness for cross diffusion is a very delicate topic (see \textit{e.g.} \cite{CJ18}), it is desirable to prove convergence to equilibrium for {\it all } solutions. This has been recently obtained in \cite{FT17a} for reaction-diffusion systems, and thus in this work, we extend these results to \textcolor{black}{reaction-cross-diffusion systems with strongly growing complex balanced reactions coming from mass-action kinetics.}

More precisely, we consider $n$ chemical substances $S_1, \ldots, S_n$ reacting via $R$ reactions of the form
{\color{black}
\begin{equation}\label{R}
	y_{r,1}S_1 + \ldots + y_{r,n}S_n \xrightarrow{k_r} y_{r,1}'S_1 + \ldots + y_{r,n}'S_n \qquad \text{ or shortly } \qquad y_r\xrightarrow{k_r} y_r', \quad \qquad r=1,\ldots, R,
\end{equation}}
where $y_r = (y_{r,1},\ldots, y_{r,n}), y_r'= (y_{r,1}', \ldots, y_{r,n}') \in (\{0\}\cup [1,\infty))^n$ are the stoichiometric coefficients, and $k_r > 0$ are the reaction rate constants. The corresponding reaction-cross-diffusion system reads for each $i=1,\ldots, n$ as
\begin{equation}\label{S}\tag{S}
	\begin{cases}
		\partial_tu_i - \mathrm{div}\left(\sum_{j=1}^{n}A_{ij}(u)\na u_j\right) = f_i(u), &\text{ for } (x,t)\in \Omega\times (0,T),\\
		\left(\sum_{j=1}^{n}A_{ij}(u)\na u_j\right)\cdot \nu = 0, &\text{ for } (x,t)\in \pa\Omega\times (0,T),\\
		u_{i}(x,0) = u_{i,0}(x), &\text{ for } x\in \Omega,
	\end{cases}
\end{equation}
where $u = (u_1, \ldots, u_n)$ are the population densities and $\Omega$ is a bounded domain with smooth boundary $\partial\Omega$, \textcolor{black}{and $\nu$ is the exterior unit normal vector to $\partial \Omega$}. The reaction terms represent the reactions in \eqref{R}, i.e.
\begin{equation}\label{reactions}
\color{black}{f_i(u) = \sum_{r=1}^{R}k_r(y_{r,i}' - y_{r,i})u^{y_r}} \quad \text{ with } \quad u^{y_r} = \prod_{i=1}^{n}u_i^{y_{r,i}},
\end{equation}
while the diffusion matrix $A(u) = [A_{ij}(u)]_{i,j=1,\ldots, n}$ is given by
\begin{equation}\label{Aij}
A_{ij}(u) = \delta_{ij}\left(a_{i0} + \sum_{k=1}^{n}a_{ik}u_k \right) + a_{ij}u_i,
\end{equation}
where $a_{i0}, a_{ij} \geq 0$ for all $i,j=1,\ldots, n$ and $\delta_{ij}$ denotes the Kronecker delta. They are assumed to satisfy (in analogy to \cite{CJ17}) either 
the weak cross-diffusion condition 
\begin{equation}\label{weak-cross}
\alpha:= \min_{i=1,\ldots,n} \left(a_{ii} - \frac14 \sum_{i=1}^n \left(\sqrt{a_{ij}} - \sqrt{a_{ji}}  \right)^2\right) > 0,
\end{equation}
or the detailed-balance condition\footnote{This should not be confused with the detailed balance condition occuring in reactions \textcolor{black}{or in even more general micro-reversible processes}. Also note that in \cite{CJ17} the detailed balance diffusion condition was $\pi_i a_{ij} = \pi_j a_{ji}$ for some positive constants $\pi_i>0$. Here we choose $\pi_i=1$ for $i=1,\ldots, n$ for the compatibility with the reactions.}
\begin{equation}\label{detail-diffusion}
a_{ij} = a_{ji} \quad \mbox{for all} ~~1\leq i,j\leq n.
\end{equation}
Let $m = \mathrm{codim}\{y_r'- y_r\}_{r=1,\ldots, R}^{\top}$, then if $m>0$ there exists a matrix $\mathbb Q \in \mathbb R^{m\times n}$ whose rows form a basis of $\mathrm{ker}\{y_r'- y_r\}_{r=1,\ldots, R} \in \mathbb R^{n\times R}$. From \eqref{reactions} it follows that $\mathbb Q [f_1(u), \ldots, f_n(u)]^{\top} = 0$, and therefore \eqref{S} formally possesses $m$ conservation laws
\begin{equation*}
	\mathbb Q\overline{u}(t) = \mathbb Q \overline{u}_0 =: \mathbf M \quad \text{ for all } \quad t>0,
\end{equation*}
where $\overline{u} = (\overline{u}_1, \ldots, \overline{u}_n)$ and $\overline{u}_i = \frac{1}{|\Omega|}\int_{\Omega}u_idx$. The system \eqref{S} is said to satisfy the \textit{complex balanced condition} if there exists a positive {\it complex balanced equilibrium} $u_\infty = (u_{1,\infty}, \ldots, u_{n,\infty})\in (0,\infty)^n$, such that at $u_\infty$ the total out-flow and in-flow at each complex are balanced, i.e.
\begin{equation}\label{com-equi}
	\sum_{\{r: y_r = y\}}k_ru_\infty^{y_r} = \sum_{\{s: y_s' = y\}}k_su_\infty^{y_s} \quad \text{ for all } \quad y\in \{y_r, y_r'\}_{r=1,\ldots, R}.
\end{equation}
It was proved in \cite{Fei} that \black{if $m>0$ then} for each positive initial mass vector $\mathbf M$ there exists a unique positive complex balanced equilibrium $u_\infty \in (0,\infty)^n$, \black{while when $m=0$ the system has a unique positive complex balanced equilibrium for any positive initial data}. Note that there could possibly exist many {\it boundary equilibria}, i.e. $u^*\in \partial(0,\infty)^n$ and $u^*$ satisfies \eqref{com-equi}.

\medskip
The main result of this paper reads as follows.
\begin{theorem}\label{thm:main}
	Let $\Omega$ be a bounded domain with smooth boundary $\partial\Omega$. Assume $a_{i0}$, $a_{ii}>0$, $a_{ij}\geq 0$, and let the diffusion matrix $A(u)$ satisfy either \eqref{weak-cross} or \eqref{detail-diffusion}. 
	Assume that \eqref{S} satisfies the complex balanced condition \eqref{com-equi}. Then, for any nonnegative measurable initial data $u_0\in L^1(\Omega)^n$ such that $u_{i,0}\log u_{i,0}\in L^1(\Omega)$ for all $i=1,\ldots, n$, there exists a global nonnegative renormalized solution $u = (u_1, \ldots, u_n)$ to \eqref{S}, that is, for all $T>0$,
	\begin{equation*}
	u_i\log u_i \in L^\infty(0,T;L^1(\Omega)), \quad \mbox{and}~~\quad\black{ \|\sqrt{u_i}\|_{L^2(0,T;H^1(\Omega))}, \|u_i\|_{L^2(0,T;H^1(\Omega))} \leq C(T)}
	\end{equation*}
	and for any smooth function $\xi \in C^\infty([0,\infty)^n)$ with compactly supported $D\xi$, it holds for all test functions $\psi \in C^\infty_0(\overline{\Omega}\times [0,T))$ that
	\begin{equation}\label{defi.renorm}
	\begin{aligned}
	-\int_{\Omega}\xi(u_0)\psi(\cdot, 0)dx - \int_0^T\int_{\Omega}\xi(u)\pa_t\psi dxdt &= -\sum_{i,k=1}^n\int_0^T\int_{\Omega}\pa_i\pa_k\xi(u)\left(\sum_{j=1}^{n}A_{ij}(u)\na u_j\right)\na u_k\psi dxdt\\
	&\quad - \sum_{i=1}^n\int_0^T\int_{\Omega}\pa_i\xi(u)\left(\sum_{j=1}^{n}A_{ij}(u)\na u_j\right)\na \psi dxdt + \sum_{i=1}^{n}\int_0^T\int_{\Omega}\pa_i\xi(u)f_i(u)\psi dxdt.
	\end{aligned}
	\end{equation}
	Assume additionally that \eqref{S} does not have any boundary equilibria and fix an initial mass vector $\mathbf M$. Then, any renormalized solution to \eqref{S} with positive initial mass $\mathbf M$, i.e. $\Q \overline{u}_0 = \mathbf M$, converges exponentially to the equilibrium, i.e.
	\begin{equation*}
		\sum_{i=1}^{n}\|u_i(t) - u_{i\infty}\|_{L^1(\Omega)} \leq Ce^{-\lambda t} \quad \text{ for all } \quad t>0,
	\end{equation*}
	where $C>0$ and $\lambda>0$ are constants which can be computed explicitly up to a finite dimensional inequality.
\end{theorem}

\begin{remark}
\textnormal{The convergence result in Theorem \ref{thm:main}, \black{in case $m>0$}, depends only on the initial masses but not on the precise initial data. Thus, two solutions with different initial data but same initial masses converge exponentially to the same equilibrium. \black{When $m=0$, i.e. there are no conservation laws, then all renormalized solutions converge to the unique positive equilibirium for any positive initial data.}}
\end{remark}

The main tool in the proof of Theorem \ref{thm:main} is to consider the relative entropy
\begin{equation}\label{re_entropy}
	\E(u|u_\infty) = \int_{\Omega}E(u|u_\infty)dx, \quad \mbox{where} \quad E(u|u_\infty) = \sum_{i=1}^n \left(u_i\log(u_i/u_{i\infty}) - u_i + u_{i\infty}\right)\geq 0,
\end{equation}
 for which formally for any solution to \eqref{S} the entropy production has the following form 
\begin{equation}\label{eep}
	\frac{d}{dt}\E(u|u_\infty) \leq  -\D(u) \quad \text{ with } \quad \D(u) = \sum_{i=1}^{n}a_{i0}\int_{\Omega}\frac{|\na u_i|^2}{u_i}dx + \sum_{r=1}^{R}k_r u_{\infty}^{y_r}\int_{\Omega}\Psi\left(\frac{u^{y_r}}{u_\infty^{y_r}},\frac{u^{y_r'}}{u_\infty^{y_r'}} \right)dx,
\end{equation}
where $\Psi(x,y) = x\log(x/y) - x + y$. For details we refer to \cite{CDJ18} for the cross-diffusion term and to \cite{FT17a} for the reaction term. Moreover, for all nonnegative measurable functions $u = (u_1, \ldots. u_n)$ satisfying the conservation laws
\begin{equation}\label{conservation}
	\mathbb Q \overline{u} = \mathbf M,
\end{equation}
it was proved (e.g. \cite{FT17a}) that
\begin{equation*}
	\D(u) \geq \lambda \E(u|u_\infty),
\end{equation*}
where $\lambda$ is an explicit constant up to a finite dimensional inequality. Then, still formally, one obtains the desired exponential decay
\begin{equation*}
	\E(u(t)|u_\infty) \leq e^{-\lambda t}\E(u_0|u_\infty).
\end{equation*}
Unfortunately, the notion of renormalized solutions is very weak, so that the entropy-entropy production inequality \eqref{eep} or even the conservation laws \eqref{conservation} (which only concern the $L^1$-norm of the solution) are  not easy to verify. As mentioned before, one can argue via approximate solutions, and thus obtain the convergence to equilibrium for {\it one } renormalized solution, see e.g. \cite{DJT18}. However, it is not clear if all  renormalized solutions (\textcolor{black}{in the sense of definition in \eqref{defi.renorm}}) can be approximated in such a way. Our aim here is to prove that {\it all } renormalized solutions with the same initial mass converge to the unique equilibrium. The main idea is to show that the conservation laws \eqref{conservation} and a weaker version of the entropy-entropy production inequality (see Lemma \eqref{weak-eep}) hold for any renormalized solution. Our proof uses the techniques developed in \cite{Fi17}.

\section{Proof of the main result}
%

\begin{lemma}[Weak entropy-entropy production inequality]\label{weak-eep}
	For any renormalized solution $u$ of \eqref{S} it holds that
	\begin{equation*}
		\E(u(t)|u_\infty) + \int_s^t\D(u(\tau))d\tau \leq \E(u(s)|u_\infty) \quad \mbox{for a.e.} \quad t>s>0,
	\end{equation*}
	where $\E$ and $\D$ are defined in \eqref{re_entropy} and \eqref{eep} respectively.
\end{lemma}
\begin{proof}
    \textcolor{black}{From this point on, we consider $C>0$ as a generic constant whose value can change from line to line, or even in the same line.}
    For $M>0$, let $\phi_M: [0,\infty) \to \mathbb{R}$ be a smooth function with 
	\begin{align}
	      \phi_M(s) &= s, \quad \mbox{for}~~s \leq M, \qquad 
	      \phi_M'(s) = 0, \quad \mbox{for}~~ s \geq M^C,\qquad 
	      \phi_M'(s) \in [0,1],\nonumber \\
	      \left|\phi_M''(s)\right| &\leq \frac{C}{1 + s \log (1+s)} \quad \mbox{for all}~~ s\geq 0\label{phi}.
	\end{align}
	Moreover, we set 
	\begin{align*}
	      \xi(u) = \phi_M(E(u + \eta|u_\infty)), 
	\end{align*}
    where $u+\eta = (u_1 + \eta, \dots, u_n + \eta)$ for some $\eta >0$. \black{The regularization $\eta>0$ is needed to deal with the potential singularity of $\log u_i$ since a renormalized solution is non-negative but in general not strictly positive.}
    For simplicity, we will write $E(u)$ and $E(u+\eta)$ instead of $E(u|u_\infty)$ and $E(u+\eta|u_\infty)$ respectively inside this proof. Then we can compute
	 \begin{align*}
	      \partial_i \xi(u) &= \phi_M'(E(u + \eta))\log\left(\frac{u_i + \eta}{u_{i\infty}}\right), \\
	      \partial_i \partial_k \xi(u) &= \phi_M''(E(u + \eta))\log\left(\frac{u_k + \eta}{u_{k\infty}}\right)\log\left(\frac{u_i + \eta}{u_{i\infty}}\right) 
	      + \phi_M'(E(u + \eta))\frac{\delta_{ik}}{u_i + \eta}.
	 \end{align*}
    By choosing $\psi = 1$ \black{, or more precisely a smooth version of $1$ with compact support in $[0,T-\delta]$ then let $\delta \to 0$ (see \cite[Lemma 11]{CJ17} for more details)} in the definition of the renormalized solutions, we get 
	\begin{align}\label{eq.star}
	      &I_1(\eta,M):=\int_{\Omega}\phi_M(E(u+\eta))\,dx\biggr|_s^t \notag \\
	      &= -\sum_{i,k=1}^n \int_s^t \int_{\Omega} \left(\phi_M''(E(u + \eta))\log\left(\frac{u_k + \eta}{u_{k\infty}}\right)\log\left(\frac{u_i + \eta}{u_{i\infty}}\right) 
	      + \phi_M'(E(u + \eta))\frac{\delta_{ik}}{u_i + \eta} \right)\left(\sum_{j=1}^n A_{ij}(u)\nabla u_j\right)\nabla u_k\,dxdt \\
	      & \quad + \sum_{i=1}^n \int_s^t \int_{\Omega} \phi'_M(E(u + \eta))\log\left(\frac{u_i + \eta}{u_{i\infty}}\right)f_i(u)\,dxdt \notag \\
	      & =: I_2(\eta,M) + I_3(\eta,M).  \notag
	\end{align}
  Our first goal now is to pass to the limit $\eta \to 0$ in \eqref{eq.star}. 
  Clearly, due to the dominated convergence theorem, we have for the left-hand side of \eqref{eq.star} that
	\begin{align*}
	      \lim_{\eta \to 0} I_1(\eta,M)  = \int_{\Omega}\phi_M(E(u))\,dx\, \biggr|_s^t.
	\end{align*}
  Next, since $\phi'_M$ has compact support {\color{black}the integrand of $I_3(\eta, M)$ vanishes when $|u|$ is large. Now for $|u| \leq C(M)$ we can use the property $f_i(u) \geq 0$ when $u_i = 0$ and the local Lipschitz continuity of $f_i(u)$ to estimate $f_i(u) \geq -C(M)u_i$. Hence by considering the signs of $f_i(u)$ and $\log(\frac{u_i+\eta}{u_{i\infty}})$ one obtains easily}
	\begin{align*}
	      f_i(u)\log\left(\frac{u_i + \eta}{u_{i \infty}}\right) \leq C(M) u_i \left|\log \frac{u_i + \eta}{u_{i \infty}}\right|.
	\end{align*}
  Thus, Fatou's lemma yields
	\begin{align*}
	      \limsup_{\eta \to 0} I_3(\eta,M)
	      \leq \sum_{i=1}^n \int_s^t \int_{\Omega} \phi'_M(E(u)) \log\left(\frac{u_i}{u_{i\infty}}\right)f_i(u)\,dxd\tau.
	       	\end{align*}
  Next, we split $I_2(\eta,M)$ in \eqref{eq.star} into 
	\begin{align*}
	    I_2(\eta,M) 
	      &=
	      -\sum_{i,k=1}^n \int_s^t \int_{\Omega} \phi''_M(E(u + \eta))\log\left(\frac{u_k + \eta}{u_{k\infty}}\right) \log\left(\frac{u_i + \eta}{u_{i\infty}}\right) \left(\sum_{j=1}^n A_{ij}(u)\nabla u_j\right)\nabla u_k \,dxd\tau \\
	      &\quad - \sum_{i=1}^n \int_s^t \int_{\Omega} \phi'_M(E(u + \eta)) \frac{1}{u_i + \eta} \left(\sum_{j=1}^n A_{ij}(u)\nabla u_j\right)\nabla u_i \,dxd\tau \\
	      &=: I_4(\eta,M) + I_5(\eta,M).
	\end{align*}
  In order to show the convergence of $I_4$, we use that $|A_{ij}(u)| \leq C\left(1 + \sum_{k=1}^n|u_k| \right)$ and \black{$\|\nabla u_j\|_{L^2(\Omega \times (0,T))} \leq C(T)$} thanks to the regularity of renormalized solutions. Then, recalling $\phi_M''$ has a compact support, we obtain by dominated convergence theorem that
      \begin{align*}
	    \lim_{\eta \to 0} I_4(\eta,M) = - \sum_{i,k=1}^n \int_s^t \int_{\Omega} \phi''_M(E(u)) \log\left(\frac{u_k}{u_{k\infty}}\right)\log\left(\frac{u_i}{u_{i\infty}}\right)\left(\sum_{j=1}^n A_{ij}(u)\nabla u_j\right)\nabla u_k\,dxd\tau.
      \end{align*}
  In a similar way, we obtain 
      \begin{align*}
	    \lim_{\eta \to 0} I_5(\eta,M) &= - \sum_{i=1}^n \int_s^t \int_{\Omega} \phi'_M(E(u))\frac{1}{u_i}\left(\sum_{j=1}^n A_{ij}(u)\nabla u_j\right)\nabla u_i \,dxd\tau.
      \end{align*}
  From \cite{CJ18} we know that if $A(u)$ satisfies \eqref{weak-cross}, then 
  \begin{equation*}
	  \sum_{i=1}^{n}\frac{1}{u_i}\left(\sum_{j=1}^n A_{ij}(u)\nabla u_j\right)\nabla u_i \geq 4\sum_{i=1}^{n}a_{i0}|\na \sqrt{u_i}|^2 + \alpha\sum_{i=1}^{n}|\na u_i|^2,
  \end{equation*}
  and if $A(u)$ satisfies \eqref{detail-diffusion}, then
  \begin{equation*}
	\sum_{i=1}^{n}\frac{1}{u_i}\left(\sum_{j=1}^n A_{ij}(u)\nabla u_j\right)\nabla u_i \geq 4\sum_{i=1}^{n}a_{i0}|\na \sqrt{u_i}|^2 + 2\sum_{i=1}^{n}a_{ii}|\na u_i|^2 + 2\sum_{i\ne j} a_{ij}|\na \sqrt{u_iu_j}|^2.
	\end{equation*}
	From both cases we infer, by noticing that $a_{i0}>0$ and $4|\na \sqrt{u_i}|^2 = |\na u_i|^2/u_i$,
	\begin{equation*}
		\lim_{\eta\to 0}I_5(\eta,M) \leq - \sum_{i=1}^n \int_s^t \int_{\Omega} \phi'_M(E(u))a_{i0}\frac{|\nabla u_i|^2}{u_i}\,dxd\tau.
	\end{equation*}  
  Putting everything together yields from \eqref{eq.star} 
      \begin{align}\label{eq.BigStar}
	    \int_{\Omega}\phi_M(E(u))\,dx\,\biggr|_s^t &\leq - \sum_{i,k=1}^n \int_s^t \int_{\Omega}\phi_M''(E(u))\log\left(\frac{u_k}{u_{k\infty}}\right)\log\left(\frac{u_i}{u_{i\infty}}\right)\left(\sum_{j=1}^n A_{ij}(u)\nabla u_j\right)\nabla u_k \,dxd\tau \\
	    &- \sum_{i=1}^n a_{i0}\int_s^t \int_{\Omega}\phi_M'(E(u))\frac{|\nabla u_i|^2}{u_i}\,dxd\tau 
	    + \sum_{i=1}^n \int_s^t \int_{\Omega}\phi'_M(E(u))\log\left(\frac{u_i}{u_{i\infty}}\right)f_i(u)\,dxd\tau \notag \\
	    &=: I_6(M) + I_7(M) + I_8(M). \notag
      \end{align}
  Our goal now is to pass to the limit $M\to \infty$ in \eqref{eq.BigStar}. For the left-hand side of \eqref{eq.BigStar}, the convergence is clear due the dominated convergence theorem. 
 For $I_7$ we can use $\sqrt{u_i} \in L^2(0,T; H^1(\Omega))$ and the dominated convergence theorem to obtain 
      \begin{align*}
	      \color{black}{\lim_{M \to \infty} I_7(M) = -\sum_{i=1}^n a_{i0}\int_s^t \int_{\Omega}  \frac{|\nabla u_i|^2}{u_i}\,dxd\tau.}
      \end{align*}
  Since $\sum_{i=1}^n \log\left(\frac{u_i}{u_{i\infty}}\right)f_i(u)\leq 0$,
  we get by Fatou's lemma that
      \begin{align*}
	      \limsup_{M \to \infty} I_8(M) \leq \sum_{i=1}^n \int_s^t \int_{\Omega} \log\left(\frac{u_i}{u_{i\infty}}\right)f_i(u)\,dxd\tau. 
      \end{align*}
      {\color{black}For $I_6(M)$ we first use the identity $\sum_{j=1}^nA_{ij}(u)\na u_j = a_{i0}\na u_i + \sum_{j=1}^na_{ij}(u_j\na u_i + u_i\na u_j)$ to estimate $I_6(M) \leq I_{61}(M) + I_{62}(M) + I_{63}(M)$ where
      \begin{equation*}
      	I_{61}(M) = C\sum_{i,k=1}^{n}\int_s^t\int_{\Omega}\left|\phi_M''(E(u))\log\left(\frac{u_k}{u_{k\infty}}\right)\log\left(\frac{u_i}{u_{i\infty}}\right)\na u_i \na u_k\right|dxd\tau,
      \end{equation*}
      \begin{equation*}
	     I_{62}(M) = C\sum_{i,j,k=1}^{n}\int_s^t\int_{\Omega}|\phi_M''(E(u))||u_j|\left|\log\left(\frac{u_k}{u_{k\infty}}\right)\na u_k\right|\left|\log\left(\frac{u_i}{u_{i\infty}}\right)\na u_i\right|dxd\tau,
      \end{equation*}
      \begin{equation*}
	      I_{63}(M) =  C\sum_{i,j,k=1}^{n}\int_s^t\int_{\Omega}\left|\phi_M''(E(u))\right|\left|\log\left(\frac{u_i}{u_{i\infty}}\right)u_i\right|\left|\log\left(\frac{u_k}{u_{k\infty}}\right)\na u_k\right||\na u_j|dxd\tau.
      \end{equation*}
		For $I_{61}(M)$ we write $\na u_i\na u_k = 4\sqrt{u_i}\sqrt{u_k}(\na\sqrt{u_i}\na \sqrt{u_k})$, then we use the property of $\phi_M''$ in \eqref{phi} to estimate
		      \begin{align*}
			    \left|\phi_M''(E(u))\log\left(\frac{u_k}{u_{k\infty}}\right)\log\left(\frac{u_i}{u_{i\infty}}\right)\right|\sqrt{u_i}\sqrt{u_k} \leq C \frac{\left|\log\left(\frac{u_k}{u_{k\infty}}\right)\right| \left|\log\left(\frac{u_i}{u_{i\infty}}\right)\right|\sqrt{u_k}\sqrt{u_i}}{1 + \sum_{j=1}^n u_j(\log(1 + u_j))^2} \leq C.
		      \end{align*}
		Hence, from the bound $\|\na \sqrt{u_i}\|_{L^2(\Omega\times(0,T))} \leq C(T)$ we obtain by dominated convergence that $\lim_{M\to+\infty}I_{61}(M) = 0$.
	To estimate $I_{62}(M)$ we have first
	\begin{align*}
		\left|\log\left(\frac{u_k}{u_{k\infty}} \right)\na u_k \right| &\leq \chi_{\{u_k \geq 1\}}\left|\log\left(\frac{u_k}{u_{k\infty}} \right)\right|\left|\na u_k \right| + 2\chi_{\{0\leq u_k \leq 1\}}\left|\log\left(\frac{u_k}{u_{k\infty}} \right)\sqrt{u_k}\right|\left|\na \sqrt{u_k} \right|\\
		&\leq \chi_{\{u_k \geq 1\}}\left|\log\left(\frac{u_k}{u_{k\infty}} \right)\right|\left|\na u_k \right| + C|\na \sqrt{u_k}|,
	\end{align*}
	and similarly
		$\left|\log\left(\frac{u_i}{u_{i\infty}} \right)\na u_i \right| \leq \chi_{\{u_i \geq 1\}}\left|\log\left(\frac{u_i}{u_{i\infty}} \right)\right|\left|\na u_i \right| + C|\na \sqrt{u_i}|.$
	Therefore
	\begin{equation*}
		I_{62}(M) \leq C\sum_{i,j,k=1}^n\int_s^t\int_{\Omega}\biggl(J_1(M)|\na u_k||\na u_i| + J_2(M)|\na u_k||\na \sqrt{u_i}| + J_3(M)|\na \sqrt{u_k}||\na u_i| + J_4(M)|\na \sqrt{u_k}||\na \sqrt{u_i}|\biggr)dxd\tau
	\end{equation*}
	with 
	\begin{equation*}
		J_1(M) = |\phi_M''(E(u))||u_j|\chi_{\{u_k \geq 1\}}\chi_{\{u_i \geq 1\}}\left|\log\left(\frac{u_i}{u_{i\infty}}\right)\right|\left|\log\left(\frac{u_k}{u_{k\infty}}\right)\right|, \quad J_4(M) = |\phi_M''(E(u))||u_j|,
	\end{equation*}
	\begin{equation*}
	J_2(M) = |\phi_M''(E(u))||u_j|\chi_{\{u_k \geq 1\}}\left|\log\left(\frac{u_k}{u_{k\infty}}\right)\right|, \quad J_3(M) = |\phi_M''(E(u))||u_j|\chi_{\{u_i \geq 1\}}\left|\log\left(\frac{u_i}{u_{i\infty}}\right)\right|.
	\end{equation*}
	Using \eqref{phi} we see that $|J_i(M)| \leq C$ for all $i=1,\ldots, 4$. Taking into account that $\|\na u_i\|_{L^2(\Omega\times(0,T))}, \|\na \sqrt{u_i}\|_{L^2(\Omega\times(0,T))}\leq C(T)$ we conclude by the dominated convergence theorem that $\lim_{M \to \infty}I_{62}(M) = 0$. The proof of $\lim_{M \to \infty} I_{63}(M) = 0$ is similar so we omit it.}
Consequently, by collecting all results together and using the fact that
\begin{equation}\label{ineq.new}
		\sum_{i=1}^nf_i(u)(\log u_i - \log u_{i\infty}) = -\sum_{r=1}^{R}k_ru_{\infty}^{y_r}\Psi\left(\frac{u^{y_r}}{u_\infty^{y_r}},\frac{u^{y_r'}}{u_\infty^{y_r'}} \right)\leq 0,
	\end{equation}
(see the computations in \cite[Proposition 2.1]{DFT17}), we obtain the desired result.
\end{proof}
\begin{lemma}[Conservation laws]\label{cons}
	When $m>0$, for any renormalized solution $u$ to \eqref{S} it holds that
	\begin{equation*}
		\Q \overline{u}(t) = \Q\overline{u}_0 \quad \text{ for all } \quad t>0.
	\end{equation*}
\end{lemma}
\begin{proof}
	 \textcolor{black}{Our proof follows from \cite[Proposition 6]{Fi17} where Fischer proved the conservation laws for reaction-diffusion systems.}
      We denote by $q = (q_1,\ldots,q_n)$ an arbitrary row of $\Q$. Thus, we have that
	$	\sum_{i=1}^n q_if_i(u) = 0. $
      Let $\phi_M$ be chosen in the same way as in the proof Lemma \ref{weak-eep}. By choosing $\xi$ as
		   $\color{black}{\xi(u) = \phi_M\left(\beta \sum_{i=1}^n q_i u_i + E(u+\eta|u_\infty)\right) }$
      where $\beta \in \mathbb{R}$ and $\psi=1$ in the definition of renormalized solutions, we can pass to the limits $\eta \to 0$ and $M \to +\infty$ like in the proof of Lemma \ref{weak-eep} to obtain
	  \begin{align*}
		\left(\beta \sum_{i=1}^n \int_{\Omega} q_i u_i \,dx + \E(u|u_\infty)\right)\,\biggr|_0^T \leq \int_0^T\D(u(\tau))\,d\tau. 
	  \end{align*}
  By dividing both sides by $\beta >0$ and letting $\beta \to +\infty$, we get that
	  \begin{align*}
		\sum_{i=1}^n \int_{\Omega} q_i u_i(T)\,dx \leq \sum_{i=1}^n \int_{\Omega}q_i u_{i0}(x)\,dx.  
	  \end{align*}
   Repeating the arguments with $\beta<0$ and letting $\beta \to -\infty$, we obtain that
		 $\sum_{i=1}^n \int_{\Omega} q_i u_i(T)\,dx \geq \sum_{i=1}^n \int_{\Omega}q_i u_{i0}(x)\,dx,$ 
   which finishes the proof of the conservation laws.

 \end{proof}

We are now ready to give the proof of the main result.
\begin{proof}[Proof of Theorem \ref{thm:main}]
	
	The existence of a global renormalized solution follows from \cite[Theorem 1]{CJ17} since under the complex balanced condition the reactions satisfy 
	\eqref{ineq.new}, which is (H4) in \cite{CJ17} with $\pi_i = 1$ and $\lambda_i = -\log u_{i\infty}$ for all $i=1,\ldots, n$.
	
	We now turn to the convergence to equilibrium. Since the system possesses no boundary equilibria, it follows from \cite[Theorem 1.1]{FT17a} that
	$	\D(u) \geq \lambda \E(u|u_\infty)$
	for all measurable nonnegative functions $u$ satisfying $\Q \overline{u} = \Q u_\infty$, where $\lambda>0$ is an explicit constant up to a finite dimensional inequality \textcolor{black}{(\cite[inequality (11)]{FT17a})}. Note that this inequality does not require any other higher regularity of $u$. Therefore, thanks to Lemma \ref{cons}, for any renormalized solution to \eqref{S} it holds
	\begin{equation*}
		\D(u(s)) \geq \lambda \E(u(s)|u_\infty) \quad \text{ for a.e. } \quad s>0.
	\end{equation*}
	Using this and Lemma \ref{weak-eep} it follows that
	\begin{equation*}
		\E(u(t)|u_\infty) + \lambda\int_s^t\E(u(\tau)|u_\infty)d\tau \leq \E(u(s)|u_\infty) \quad \text{ for a.e. } \quad t>s.
	\end{equation*}
	By Gronwall's inequality we get
	\begin{equation*}
		\E(u(t)|u_\infty) \leq e^{-\lambda t}\E(u_0|u_\infty),
	\end{equation*}
	and a Csisz\'ar-Kullback-Pinsker type inequality (see e.g. \cite[Lemma 2.2]{FT17a}) completes the proof of Theorem \ref{thm:main}.
\end{proof}

\par{\bf Acknowledgements:} Both authors would like to thank Prof. Ansgar J\"ungel for the fruitful discussions. The first author acknowledges partial support from the Austrian Science Fund (FWF), grants P27352 and P30000, while the second author is partially supported by the International Training Program IGDK 1754 and NAWI Graz.


\begin{thebibliography}{11}
%


\bibitem{CDF14} J.~A.~Ca\~{n}izo, L.~Desvillettes, and K.~Fellner. Improved duality estimates and applications to reaction-diffusion  equations. \textit{Comm. Partial Differential Equations} 39 (2014) no.6, 1185--1204. 

\bibitem{CDJ18} X.~Chen, E.~S.~Daus, and A.~J\"ungel. Global existence analysis of cross-diffusion population systems for multiple species. \textit{Arch. Ration. Mech. Anal.} 227 (2018), no. 2, 715--747.

\bibitem{ChJu04} L.~Chen and A.~J\"ungel. Analysis of a multi-dimensional
parabolic population model with strong cross-diffusion.
{\em SIAM J. Math. Anal.} 36 (2004), 301--322.

\bibitem{ChJu06} L.~Chen and A.~J\"ungel. Analysis of a parabolic cross-diffusion 
population model without self-diffusion. {\em J. Diff. Eqs.} 224 (2006), 39--59. 


\bibitem{CJ17} X.~Chen and A.~J\"ungel. Global renormalized solutions to reaction-cross-diffusion
systems. {\em arXiv:1711.01463}.

%

\bibitem{CJ18} X.~Chen and A.~J\"ungel. A note on the uniqueness of weak solutions to a class of cross-diffusion systems. To appear in {\em J. Evol. Eqs.}, 2018.

\bibitem{CJ18b} X.~Chen and A.~J\"ungel. Weak-strong uniqueness of renormalized solutions to reaction-cross-diffusion systems. {\em arXiv:1805.02950v1}.

\bibitem{DJT18} E.~S.~Daus, A.~J\"{u}ngel, and B.~Q.~Tang. Exponential time decay of solutions to reaction-cross-diffusion systems of Maxwell-Stefan type. {\em arXiv:1802.10274}.
%
\bibitem{DeFe06} L.~Desvillettes and K.~Fellner. Exponential decay toward equilibrium
via entropy methods for reaction-diffusion equations. {\em J. Math. Anal. Appl.}
319 (2006), 157--176.
%
\bibitem{DFT17} L.~Desvillettes, K.~Fellner, and B.~Q.~Tang. Trend to equilibrium for reaction-diffusion systems arising from complex balanced chemical reaction networks. {\em SIAM J. Math. Anal.} 49 (2017), 2666--2709.
%
\bibitem{DeFe14} L.~Desvillettes and K.~Fellner. Exponential convergence to
equilibrium for nonlinear reaction-diffusion systems arising in reversible
chemistry. In: C.~P\"otzsche, C.~Heuberger, B.~Kaltenbacher, and F.~Rendl (eds.). 
{\em System Modeling and Optimization}. CSMO 2013, IFIP Advances in Information and 
Communication Technology, vol.\ 443, pp.\ 96--104. Springer, Berlin, 2014.

\bibitem{DLM14} L.~Desvillettes, T.~Lepoutre, and A.~Moussa. Entropy, duality,
and cross diffusion. {\em SIAM J. Math. Anal.} 46 (2014), 820-853.

\bibitem{DLMT15} L.~Desvillettes, T.~Lepoutre, A.~Moussa, and A.~Trescases. 
On the entropic structure of reaction-cross diffusion systems. 
{\em Commun. Partial Diff. Eqs.} 40 (2015), 1705-1747.

\bibitem{Fei} M.~Feinberg. The existence and uniqueness of steady states for a
class of chemical reaction networks. {\em Arch. Rational Mech. Anal.} 132 (1995),
311--370.

\bibitem{Fi15} J.~Fischer. Global existence of renormalized solutions to entropy-dissipating reaction-diffusion systems. {\em Arch. Rational Mech. Anal.} 218 (2015), no.1, 553--587.

\bibitem{Fi17} J.~Fischer. Weak-strong uniqueness of solutions to entropy-dissipating reaction-diffusion equations.\textit{ Nonlinear Anal.} 159 (2017), 181--207.

\bibitem{FLT17} K.~Fellner, E.~Latos, and B.~Q.~Tang. Global regularity and convergence to equilibrium of reaction-diffusion systems with nonlinear diffusion. \textit{arXiv:1711.02897}.
%
\bibitem{FT17a} K.~Fellner and B.~Q.~Tang. Convergence to equilibrium of 
renormalised solutions to nonlinear chemical reaction-diffusion systems.
{\em Z. Angew. Math. Phys.}, 69.3 (2018).
%
%
%
%
%
%
%
%
%
%
%
\bibitem{Jue15} A.~J\"ungel. The boundedness-by-entropy method for cross-diffusion
systems. {\em Nonlinearity} 28 (2015), 1963--2001.
%
\bibitem{Jue16} A.~J\"ungel. {\em Entropy Methods for Diffusive Partial
Differential Equations}. BCAM SpringerBriefs, 2016.
%
%
%
%
%
%
%
%
%

\bibitem{SKT79} N.~Shigesada, K.~Kawasaki, and E.~Teramoto. Spatial segregation of 
interacting species. {\em J. Theor. Biol.} 79 (1979), 83-99.
%
%
%
%

\end{thebibliography}
\end{document}